\documentclass[11pt]{article}
\usepackage{tikz}
\usepackage{ucs}
\usepackage{amssymb}
\usepackage{amsthm}
\usepackage{amsmath}
\usepackage{latexsym}
\usepackage[cp1251]{inputenc}
\usepackage[english]{babel}
\usepackage{graphicx}
\usepackage{wrapfig}
\usepackage{caption}
\usepackage{subcaption}
\usepackage{txfonts}
\usepackage{lmodern}
\usepackage{mathrsfs}
\usepackage{algorithm}
\usepackage{multicol}
\usepackage{enumerate}
\usepackage{titlesec}
\usepackage{paralist}
\usepackage{mathtools}

\titleformat{\paragraph}
{\normalfont\normalsize\bfseries}{\theparagraph}{1em}{}
\titlespacing*{\paragraph}
{0pt}{3.25ex plus 1ex minus .2ex}{1.5ex plus .2ex}

\newcommand*{\myproofname}{Proof}

\usepackage{wasysym}

\usepackage{fullpage}

\def\qed{\hfill\ifhmode\unskip\nobreak\fi\qquad\ifmmode\Box\else\hfill$\Box$\fi}

\title{Cubic graphs with small independence ratio }
\date{\today}
\author{
J\' ozsef Balogh~\thanks{Department of Mathematics, University of Illinois at Urbana--Champaign, IL, USA, jobal@illinois.edu. 
Research of this author is partially supported by NSF Grant DMS-1500121, Arnold O. Beckman
Research Award (UIUC Campus Research Board 15006) and by the Langan Scholar Fund
(UIUC).}
\and Alexandr Kostochka \thanks{Department of Mathematics, University of Illinois at Urbana--Champaign, IL, USA and
Sobolev Institute of Mathematics, Novosibirsk 630090, Russia, kostochk@math.uiuc.edu. Research of this author is supported in part by NSF grant
 DMS-1600592 and by grants 15-01-05867  and 16-01-00499 of the Russian Foundation for Basic Research.}
 \and Xujun Liu\thanks{Department of Mathematics, University of Illinois at Urbana--Champaign, IL, USA,  xliu150@illinois.edu.}
 }

\makeatletter
\newcommand{\neutralize}[1]{\expandafter\let\csname c@#1\endcsname\count@}
\makeatother

\newtheorem{theo}{Theorem}

\newtheorem{lemma}[theo]{Lemma}

\newtheorem{corl}[theo]{Corollary}

\newtheorem{claim}{Claim}[theo]

\theoremstyle{definition}
\newtheorem{defn}[theo]{Definition}

\newtheorem{remk}[theo]{Remark}

	\def\quotient#1#2{%
		\raise1ex\hbox{$#1$}\Big/\lower1ex\hbox{$#2$}%
	}

	\renewcommand{\epsilon}{\varepsilon}

\begin{document}
	\maketitle
	
\begin{abstract}
Let $i(r,g)$ denote the infimum of the ratio $\frac{\alpha(G)}{|V(G)|}$
 over the $r$-regular graphs of girth at least $g$, where $\alpha(G)$ is the independence number of $G$,
  and  let $i(r,\infty) := \lim\limits_{g \to \infty} i(r,g)$. 
  Recently, several new lower bounds of $i(3,\infty)$ were obtained. In particular, Hoppen and Wormald showed in 2015 that $i(3, \infty) \ge 0.4375,$ and Cs\'oka improved it to $i(3,\infty) \ge 0.44533$ in 2016.
  Bollob\'as proved the upper bound  $i(3,\infty) < \frac{6}{13}$  in 1981,   and McKay improved it to $i(3,\infty) < 0.45537$   in 1987. 
There were no improvements since then.
 In this paper, we improve the upper bound to $i(3,\infty) \le 0.454.$
\end{abstract}

	\section{Introduction}	

A set $S$ of vertices in a graph $G$ is  {\em independent }   if no two vertices of $S$ are joined by an edge.
The  {\em independence number}, $\alpha(G)$, is the maximum size of an independent set in $G$. The 
{\em independence ratio}, $i(G)$, of a graph $G$ is the ratio $\frac{\alpha(G)}{|V(G)|}$. For positive integers $r$ and $g$, $i(r,g)$ denotes 
 the infimum of $i(G)$ over the $r$-regular graphs of girth at least $g$, and $i(r,\infty)$ denotes $\lim\limits_{g\to\infty} i(r,g)$.
The first interesting upper bounds on $i(r,\infty)$ were obtained by Bollob\' as~\cite{B1} in 1981. In particular, he proved
$i(3,\infty)<\frac{6}{13}$. Refining the method, McKay~\cite{MK1} in 1987 showed
\begin{theo}[McKay~\cite{MK1}]\label{McKay}
\begin{equation}\label{e2}
i(3,\infty)<0.45537.
\end{equation}
\end{theo}
In the next 30 years, there were no improvements of Theorem~\ref{e2}, but recently some interesting lower bounds on
$i(r,\infty)$ and in particular on $i(3,\infty)$ were proved. Hoppen~\cite{H1} showed $i(3,\infty)\geq 0.4328$. Then
Kardo\v s, Kr\' al and Volec~\cite{KKV1} improved the bound to $0.4352$. Cs\' oka, Gerencs\' er, Harangi, and Vir\' ag~\cite{CGHV1} pushed the bound to
$0.4361$ and Hoppen and Wormald~\cite{HW1} --- to $0.4375$. Moreover, Cs\' oka et al~\cite{CGHV1} claimed a computer
assisted lower bound $i(3,\infty)\geq 0.438$, and Cs\' oka~\cite{C1} later improved the bound to 0.44533. The lower bound of $i(3, \infty)$ was also studied in~\cite{HS1},~\cite{LW1} and~\cite{S1}.
Our result is an improvement of~\eqref{e2}
to $i(3,\infty) \le 0.454$. The improvement is  small, but it decreases the gap between the upper and lower bounds on
$i(3,\infty)$ by approximately $14\%$.

\begin{theo}\label{main theorem}
$i(3,\infty) \le 0.454.$
\end{theo}

The proof uses the  language of configurations introduced by Bollob\' as~\cite{B2}, and shows that ``many'' 3-regular configurations have ``small''
independence ratio. The proof of our improvement is based on analyzing the presence not of  largest independent sets, but of larger structures, so called MAI-sets (defined in
Section 3) that
contain largest independent sets.

 
\section{Preliminaries}\label{pr}
\subsection{Notation}
We mostly use standard notation.
 The complete $n$-vertex graph is denoted by $K_n$.
  If $G$ is a multigraph and $v,u \in V(G)$, then $E_G(v, u)$ denotes the set of all edges in $G$ connecting  $v$ and $u$,
 $e_G(v, u) \coloneqq |E_G(v, u)|$, and $\deg_G(v)\coloneqq\sum_{u\in V(G)\setminus \{v\}} e_G(v,u)$. 
 By $\Delta(G)$ we denote the maximum degree of $G$,  and by $g(G)$ --- the {\em girth} (the length of a shortest
 cycle) of $G$.
For $A \subseteq V(G)$, $G[A]$ denotes the submultigraph of $G$ induced by $A$.
For $k \in \mathbb{Z}_{> 0}$, $[k]$ denotes the set $\{1, \ldots, k\}$.


\subsection{The Configuration Model}
The configuration model  in different versions is due to Bender and Canfield~\cite{BC1} and Bollob\'as~\cite{B2}. 
Our work is based on the version of Bollob\'as. 
Let $n$ be an even positive integer and
 $V_n=[n]$.  Consider the Cartesian product $W_n=V_n \times [3]$. 
A {\em configuration/pairing} (of order $n$ and degree $3$) is 
a perfect matching on the vertex set $W_n$. There are 
$(3n-1)\cdot (3n-3)\cdot \ldots \cdot 1=(3n-1)!!$
such matchings.

Let $\mathcal{F}_3(n)$  denote the collection of all $(3n-1)!!$ possible pairings on $W_n$. 
 We project each pairing $F \in \mathcal{F}_3(n)$ to a multigraph $\pi(F)$ on the vertex set $V_n$ by ignoring the second coordinate. 
Then $\pi(F)$ is a $3$-regular multigraph (which may or may not contain loops and/or multiple edges). Let $\pi(\mathcal{F}_3(n))=\{\pi(F)\,:\, F\in\mathcal{F}_3(n)\}$ 
be the set of $3$-regular multigraphs on $V_n$. By definition,
\begin{equation}\label{0311}
\mbox{\em each simple graph $G\in \pi(\mathcal{F}_3(n))$ corresponds to  $(3!)^n$ distinct pairings in $\mathcal{F}_3(n)$.} 
\end{equation}
We will call the elements of $V_n$ - {\em vertices}, and of $W_n$ - {\em points}.

\begin{defn}
Let $\mathcal{G}_g(n)$ be the set of all cubic graphs with vertex set $V_n=[n]$ and girth at least $g$ and 
$\mathcal{G}'_g(n)=\{F\in \mathcal{F}_3(n)\,:\, \pi(F)\in \mathcal{G}_g(n)\}$.
\end{defn}

We will heavily use the following result: 
\begin{theo}[Wormald~\cite{W1}, Bollob\'as~\cite{B2}]\label{BW}
For each fixed $g \geq 3$, 
\begin{equation}\label{03112}
\lim\limits_{n \to \infty}  \frac{|\mathcal{G}'_g(n)|}{|\mathcal{F}_3(n)|} = \exp \left \{ -\sum\limits_{k=1}^{g-1}\frac{2^{k-1}}{k} \right \}.
\end{equation}

\end{theo}

{\bf Remark.}
When we say that {\em a pairing $F$ has a multigraph property $\mathcal{A}$}, we  mean that $\pi(F)$
has  property $\mathcal{A}$. 

Since dealing with pairings is simpler than working with labeled simple regular graphs, we need the following 
well-known consequence of Theorem~\ref{BW}.

\begin{corl}[\cite{MK1}(Corollary 1.1),~\cite{JLR}(Theorem 9.5)
]\label{MK2}
 For fixed $g \geq 3$, any property that holds for $\pi(F)$ for almost all pairings  $F\in \mathcal{F}_3(n)$
 also holds for almost all graphs in $\mathcal{G}_g(n)$. 
\end{corl} 

\begin{defn}

For a graph $G$, let $I(G)$ denote the total number of all independent sets in $G$, including the empty set. 
For all integer $r \ge 0$, $g \ge 3$, we define $I(r,g) = \inf I(G)^{1/|V(G)|}$, where the infimum is over all graphs $G$ of maximum degree at most $r$ and girth at least $g$.

\end{defn}

Recall that the {\em Fibonacci numbers} $F_n$ 
 are defined by  $F_1 = F_2 = 1,$ and $F_i = F_{i-1} + F_{i-2}$, for $i \ge 3.$ The exact formula for $F_i$ is 
$$F_i = \frac{\varphi^i - \psi^i}{\sqrt{5}},$$
where $i \ge 0$, $\varphi = \frac{1+\sqrt{5}}{2}$, and $\psi = \frac{1-\sqrt{5}}{2}$.

\begin{lemma}[McKay~\cite{MK1}]\label{GR}
For any $g \ge 4,$ $I(2,g) = (F_{s-1} + F_{s+1})^{\frac{1}{s}}$, where $s = 2 \lfloor g/2 \rfloor + 1.$ 

\end{lemma}

\begin{remk}\label{rm8}
The numbers $s-1$ and $s+1$ in Lemma~\ref{GR}   are even. Therefore, 
$$I(2,g) = (F_{s-1} + F_{s+1})^{\frac{1}{s}} = \left(\frac{\varphi^{s-1} + \varphi^{s+1}-\varphi^{1-s} - \varphi^{-s-1}}{\sqrt{5}}\right)^{\frac{1}{s}}$$
$$ = \varphi \cdot \left((1-\varphi^{-2s})\frac{\varphi^{-1}+\varphi}{\sqrt{5}}\right)^{\frac{1}{s}} = \varphi (1-\varphi^{-2s})^{1/s}.$$ 
Since the function $(1-\varphi^{-2s})^{1/s}$ monotonically increases for $s\geq 1$, and 
$\varphi (1-\varphi^{-18})^{1/9}\geq 1.618002$,
we conclude that 
 for each graph $H$ with maximum degree at most $2$ and girth at least  $ 8,$
\begin{equation}\label{1.618}
1.618 \le I(2,8) \le I(H)^{1/|V(H)|} .
\end{equation}
\end{remk}

\section{MAI sets in cubic graphs}
\begin{defn}
A vertex set $A$ in a graph $G$ is  {\em an AI set} (an almost independent  set), if every component of $G[A]$ is an edge or an isolated vertex.
In other words, $A$ is  {\em an AI set} if $\Delta(G[A])\leq 1$.
\end{defn}

\begin{defn}
A vertex set $A$ is a {\em maximum almost independent  set (MAI set)} in a graph $G$ if all of the following hold:
\begin{itemize}
\item[M1.] $A$ is an AI set;

\item[M2.]  $A$ contains  an independent set $A'$ of size $\alpha(G)$;

\item[M3.]  $A$ is  largest among all sets  satisfying M1 and M2.
\end{itemize}
\end{defn}

Let $G\in \mathcal{G}_{16}(n)$ and $A$ be a MAI set. Denote $B=V(G)-A$.

\begin{lemma}\label{B1}
 $B$ is an AI set.
\end{lemma}	

\noindent \noindent  {\em Proof.} Let $b \in B$. We prove that $d_{G[B]}(b)\leq 1$. Let $A'$ be a maximum independent set in $A$.

 If $d_{G[B]}(b)=3$, then there is no edge from $b$ to $A$,  and  $A' \cup \{b\}$ is an independent set in $G$ with size 
$|A'| + 1 =\alpha(G)+1$,  contradicting the definition of  $\alpha(G)$.

 If $d_{G[B]}(b)=2$, then there is only one edge $e$ from $b$ to $A$, say $ba$. 
 If $d_{G[A]}(a)=0$, then  $G[A \cup \{b\}]$ is an AI set in $G$ larger than $A$ containing $A'$. This contradicts the fact that $A$ is a MAI set.
  If $d_{G[A]}(a)=1$, then without loss of generality, we may assume $a \in A-A'$. Then $b$ has no neighbors in $A'$,
  and $A' \cup \{b\}$ is an independent set in $G$ with size $|A'| + 1 $, again contradicting the definition of  $\alpha(G)$.\qed


\medskip
Let $A$ be a MAI set in $G \in G_{16}(n)$. Denote the set of vertices with degree $1$ in $G[A]$ by $Y$, the set of vertices with degree $1$ in $G[B]$ by $Z$. We introduce notation for the sizes of the sets: Let $x:=|A'|$, $s:=|Y|/2$, $t:=|Z|/2$, and $i:=\frac{n}{2}-|A|.$ Then $|A|= \frac{n}{2}-i$ and $|B|=\frac{n}{2}+i$.

\begin{lemma}\label{its}
$i \ge 0$ and $t \ge s$.
\end{lemma}

\noindent \noindent  {\em Proof.} 
We count the number of edges with one end in $A$ and one end in $B$ in two ways. We have 
\begin{equation}\label{cut}
2s \cdot 2 + \left(\frac{n}{2}-i-2s\right) \cdot 3 = e[A,B] = 2t \cdot 2 + \left(\frac{n}{2}+i-2t\right) \cdot 3,
\end{equation}
 i.e.,
\begin{equation}\label{ts}
t-s = 3i.
\end{equation}
We also know that $x=\alpha(G)$, so 
$$x = \frac{n}{2}-i-s \ge \frac{n}{2}+i-t,$$ i.e.,
$$2i \le t-s = 3i,$$
which implies that
$$i \ge 0 \text{ and } t \ge s.\qed$$

\begin{lemma}\label{YZ} If $G\in \mathcal{G}_5(n)$, 
 then\\
(i) each vertex in $Z$ has degree at most one to $Y$; \\
(ii) each vertex in $Y$ has degree at most one to $Z$.
\end{lemma}

\begin{proof}
(i) Suppose $z \in Z$ and  $N_G(z)=\{z', y_1,y_2\}$, where $z'\in Z$ and   $y_1,y_2\in Y$.
 Since  $g(G)\geq 4$,  $y_1\neq y_2$, $y_1y_2\notin E(G)$, and so $A-y_1-y_2$ contains an independent 
 set $A'$ with $|A'|=\alpha(G)$.
 Thus  the set $A'+z$ is an independent set of size $\alpha(G)+1$ contradicting the definition of $\alpha(G)$.

(ii) Similarly, suppose $y \in Y$ and  $N_G(y)=\{y', z_1,z_2\}$, where $y'\in Y$ and   $z_1,z_2\in Z$.
Then $A-y$ contains an independent 
 set $A'$ with $|A'|=\alpha(G)$.
For $i=1,2$, let $N_G(z_i)=\{z'_i, y,a_i\}$, where $z'_i\in Z$. By Part (i), $a_1,a_2\notin Y$. Since $g(G)\geq 5$,
$a_2\neq a_1$. Then  $(A-y) \cup \{z_1,z_2\}$ is an AI set containing $A'$ and is larger than $A$, a contradiction.
\end{proof}



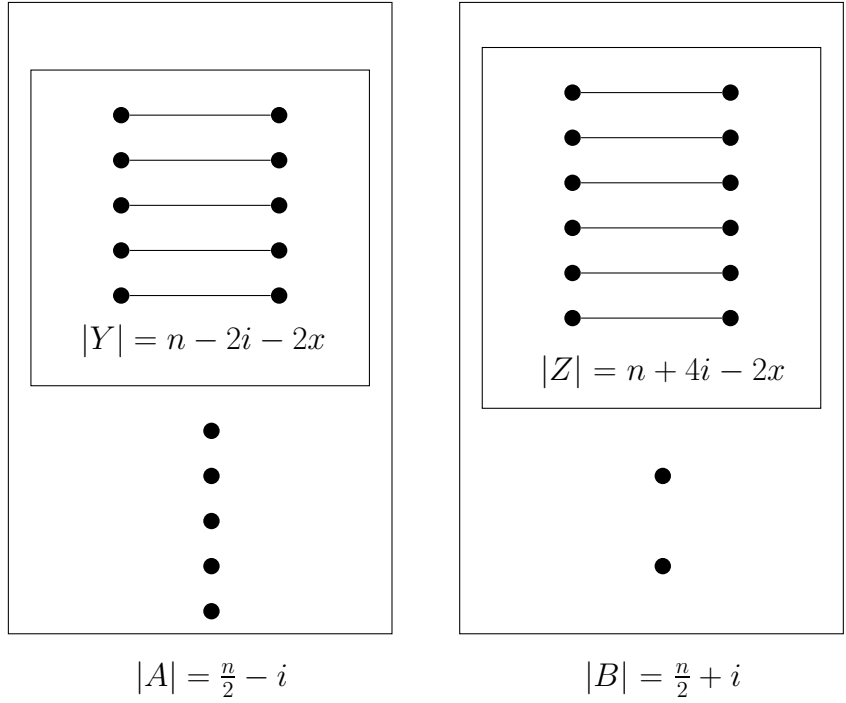
\begin{figure}[ht]\label{f1}
\begin{center}
\begin{tikzpicture}[scale=0.6, transform shape]

\draw  (-5.5,6) rectangle (3,-8);

\draw  (4.5,6) rectangle (13,-8);
\node [ shape=circle, minimum size=0.1cm,  fill = black!1000, align=center] (v1) at (-3,3.5) {};
\node [ shape=circle, minimum size=0.1cm,    fill = black!1000, align=center] (v2) at (0.5,3.5) {};
\node [ shape=circle, minimum size=0.1cm,    fill = black!1000, align=center] (v3) at (-3,2.5) {};
\node [  shape=circle, minimum size=0.1cm,  fill = black!1000, align=center] (v4) at (0.5,2.5) {};
\node [shape=circle, minimum size=0.1cm,    fill = black!1000,  align=center] (v5) at (-3,1.5) {};
\node [ shape=circle, minimum size=0.1cm,   fill = black!1000,   align=center] (v6) at (0.5,1.5) {};
\node [  shape=circle, minimum size=0.1cm,   fill = black!1000,   align=center] (v7) at (-3,0.5) {};
\node [ shape=circle, minimum size=0.1cm,    fill = black!1000,   align=center] (v8) at (0.5,0.5) {};
\node [  shape=circle, minimum size=0.1cm,    fill = black!1000,   align=center] (v9) at (-3,-0.5) {};
\node [  shape=circle, minimum size=0.1cm,    fill = black!1000,   align=center] (v10) at (0.5,-0.5) {};

\node [  shape=circle, minimum size=0.1cm,    fill = black!1000,  align=center] at (-1,-3.5) {};
\node [  shape=circle, minimum size=0.1cm,   fill = black!1000,  align=center] at (-1,-4.5) {};
\node [  shape=circle, minimum size=0.1cm,   fill = black!1000,   align=center] at (-1,-5.5) {};
\node [  shape=circle, minimum size=0.1cm,   fill = black!1000,   align=center] at (-1,-6.5) {};
\node [ shape=circle, minimum size=0.1cm,   fill = black!1000,    align=center] at (-1,-7.5) {};
\node [ shape=circle, minimum size=0.1cm,   fill = black!1000,    align=center] (v11) at (7,4) {};
\node [  shape=circle, minimum size=0.1cm,   fill = black!1000, align=center] (v12) at (10.5,4) {};
\node [ shape=circle, minimum size=0.1cm,    fill = black!1000,    align=center] (v13) at (7,3) {};
\node [  shape=circle, minimum size=0.1cm,    fill = black!1000,    align=center] (v14) at (10.5,3) {};
\node [  shape=circle, minimum size=0.1cm,   fill = black!1000,   align=center] (v15) at (7,2) {};
\node [ shape=circle, minimum size=0.1cm,   fill = black!1000,    align=center] (v16) at (10.5,2) {};
\node [ shape=circle, minimum size=0.1cm,    fill = black!1000,    align=center] (v17) at (7,1) {};
\node [ shape=circle, minimum size=0.1cm,  fill = black!1000,  align=center] (v18) at (10.5,1) {};
\node [  shape=circle, minimum size=0.1cm,   fill = black!1000,    align=center] (v19) at (7,0) {};
\node [  shape=circle, minimum size=0.1cm,   fill = black!1000,    align=center] (v20) at (10.5,0) {};

\node [  shape=circle, minimum size=0.1cm,   fill = black!1000,  align=center] at (9,-4.5) {};

\node [ shape=circle, minimum size=0.1cm,   fill = black!1000,   align=center] at (9,-6.5) {};
\draw  (v1) edge (v2);
\draw  (v3) edge (v4);
\draw  (v5) edge (v6);
\draw  (v7) edge (v8);
\draw  (v9) edge (v10);
\draw  (v11) edge (v12);
\draw  (v13) edge (v14);
\draw  (v15) edge (v16);
\draw  (v17) edge (v18);
\draw  (v19) edge (v20);

\node [ font={\huge\bfseries}, shape=rectangle, minimum size=1cm, text=black, very thick, draw=white!0, top color=white,bottom color=green!0, align=center] at (-1,-9) {$|A|=\frac{n}{2}-i$};
\draw  (-5,4.5) rectangle (2.5,-2.5);

\node [ font={\huge\bfseries}, shape=rectangle, minimum size=1cm, text=black, very thick, draw=white!0, top color=white,bottom color=green!0, text width=4cm, align=center] at (9,-9) {$|B|=\frac{n}{2}+i$};
\node [ font={\huge\bfseries}, shape=rectangle, minimum size=1cm, text=black, very thick, draw=white!0, top color=white,bottom color=green!0, text width=5.8cm, align=left] at (-1,-1.5) {$|Y|=n-2i-2x$};
\node [ font={\huge\bfseries}, shape=rectangle, minimum size=1cm, text=black, very thick, draw=white!0, top color=white,bottom color=green!0, text width=5.8cm, align=center] at (9,-2.2) {$|Z|=n+4i-2x$};

\draw  (5,5) rectangle (12.5,-3);
\node [shape=circle, minimum size=0.1cm, fill = black!1000,   align=center] (v21)  at (7,-1) {};
\node [ shape=circle, minimum size=0.1cm, fill = black!1000,   align=center] (v22) at (10.5,-1) {};
\draw  (v21) edge (v22);

\end{tikzpicture}
\caption{A MAI set $A$.}
\end{center}
\end{figure}

Let $J=\{y_1z_1,\ldots,y_jz_j\}$ be the set of all edges connecting $Y$ with $Z$ in $G$. By Lemma~\ref{YZ}, $J$ is a matching in $G$.
Define an auxiliary graph $H=H(A)$ as follows: $V(H)=J$, and $y_{\ell} z_{\ell}$ is adjacent to $y_{\ell'} z_{\ell'}$ if
$y_{\ell} y_{\ell'}\in E(G)$ or $z_{\ell} z_{\ell'}\in E(G)$. By construction, the maximum degree of $H$ is at most $2$ and a cycle of length $c$ in $H$ corresponds to a cycle of length $2c$ in $G$.

\begin{lemma}\label{J1}
The graph $G$ contains at least $I(H)$ distinct MAIs.
\end{lemma}

\begin{proof} Let $J'=\{y_1z_1,\ldots,y_{j'}z_{j'}\}$ be an arbitrary independent set in $H$.
Then the sets $Y_1=\{y_1,\ldots,y_{j'}\}$ and $Z_1= \{z_1,\ldots,z_{j'}\}$ are independent in $G$.
By the definition of $Y$, $A-Y_1$ contains an independent set $A'$ with $|A'|=\alpha(G)$.
Let $A_1=(A- Y_1)\cup Z_1$. By Lemma~\ref{YZ}, the  degree in $G[A_1]$ of every vertex in $(Y-Y_1)\cup Z_1$ is at most $1$.
If a vertex $a\in A-Y$ is adjacent to two vertices, say $z_1,z_2$ in $Z_1$, then the set $(A'-a)\cup \{z_1,z_2\}$ is independent and
is larger than $A'$, a contradiction. Thus, $A_1$ is an AI set.
Since $|A_1|=|A|$, this proves the lemma.
\end{proof}



\begin{remk}
Recall that $|A| = \frac{n}{2} - i$, $|B|=\frac{n}{2} + i,$ $|Y| = 2s = 2 (\frac{n}{2}-i-x),$ and $|A-Y| = 2x - \frac{n}{2} + i.$  
By~\eqref{ts}, we know that $t=3i+s= \frac{n}{2} + 2i - x.$ Therefore, $|Z| = 2t= 2 (\frac{n}{2}+2i-x)$ and $|B-Z| =2x - \frac{n}{2} - 3i$. 
By~\eqref{cut},  $e[A,B] = 2x + \frac{n}{2} - i.$
\end{remk}

\section{The set up of the proof}
\subsection{Restating the theorem}
We will use Theorem~\ref{McKay} of McKay in the following stronger form.

\begin{theo}[McKay~\cite{MK1}]\label{MK1}
 For every $\epsilon > 0$, there exists an $N>0$ such that for each $n>N$,  
 $${|\{ F |  F \in \mathcal{F}_3(n)\,:\;  \alpha(\pi(F))>0.45537n\}|} \ <\epsilon \cdot {(3n-1)!!}.$$
\end{theo}

We will show that  ``almost all" cubic labeled graphs of girth at least $16$ have independence ratio
at most $0.454$. In view of Theorem~\ref{BW}, the following more technical statement implies
 Theorem~\ref{main theorem}.

\begin{theo}\label{solve}
For every $\epsilon > 0,$ there is an $N > 0$ such that for each $n > N,$
\begin{equation} \label{d}
{|\{  F \in \mathcal{G}'_{16}(n)\, :\;   \alpha(\pi(F))>0.454n\}|} \ <\epsilon\, (3n-1)!!.
\end{equation}
\end{theo}

A referee asked whether one can derive from Theorem~\ref{solve}
that
{\em a random cubic graph $G \in G_{n,3}$ asymptotically almost surely satisfies $\alpha(G) \le 0.454n$}. 
We do not see how to derive this from the statement of the theorem  but think that one can modify our proof to show this fact.


The rest of the paper is a proof  of Theorem~\ref{solve}. 
 By definition, every graph has a MAI set. So, for large $n$, nonnegative integers $x\geq 0.454n$ and $i\leq \frac{n}{2}-x$, 
and each set $A$ of size $\frac{n}{2}-i$ with a fixed matching of size $\frac{n}{2}-i-x$
we will estimate the total
{\em $x$-weight} of configurations $F\in \mathcal{G}'_{16}(n)$ in which $A$ forms a MAI set. The idea of the weight (used by McKay in~\cite{MK1}) is  
to decrease overcount of the configurations containing a given
MAI set, but guarantee that the total weight of each configuration containing at least one  MAI set with independence number $x$  would be at least $1$.

\subsection{Setup of the proof of Theorem~\ref{solve}}
An {\em AI-pair} on  $[n]$ is 
 a pair  $(A,R)$ consisting of a set $A\subset [n]$ and a matching $R$ on a subset of $A$ such that $E(G[A]) = R.$
 The {\em independence number}, $\alpha(A,R)$, of an AI-pair $(A,R)$ is $|A|-|R|$.
 Let $\mathcal{P}(n,x)$ denote the family of all AI-pairs $(A,R)$ on $[n]$ with $\alpha(A,R)=x$.
 
A {\em preimage} of an AI-pair $(A,R)$ on $[n]$ is a pair $(\hat A,\hat R)$ where $\hat A=A\times [3]$ and $\hat R$ is
a matching on a subset of $\hat A$ with $|\hat R|=|R|$ such that for each edge $(i,j)(i',j')\in \hat R$,
$ii'\in R$. In other words, each edge  $e\in R$ is obtained from an edge in $\hat e\in\hat R$ by ignoring the second coordinates of the ends of $\hat e$, and this mapping is one-to-one.

 By  the $x$-weight of a configuration $F$
we mean 
\begin{equation}\label{weight}
\parbox{13.5cm}{$\omega_x(F):=$ the reciprocal of the number of preimages $(\hat A,\hat R)\subseteq F$ of 
AI-pairs $(A,R)$ on $[n]$
such that $A$ is an AI set in $\pi(F)$ with $E(\pi(F)[A])=R$ and $\alpha(A,R)=x$.}
 \end{equation}
 By the definition of $x$-weight, each pairing $F\in \mathcal{G}'_{16}(n)$ with $\alpha(\pi(F))=x$
 contributes exactly $1$ to
 \begin{equation}\label{j25}
\sigma(n,x,16):=\sum_{(A,R)\in \mathcal{P}(n,x)}\{\omega_x(F'): \mbox{$F'\in \mathcal{G}'_{16}(n)$ and $(\hat A,\hat R)$ \em is an induced subpairing of $F'$}\}.
\end{equation}
 It follows that
  \begin{equation}\label{j251}
\sigma(n,x,16)\geq \left|\{\mbox{$F'\in \mathcal{G}'_{16}(n)$  with $\alpha(\pi(F'))=x$}\}\right|.
\end{equation}

\begin{lemma}\label{formula}
Let $n$ be a positive even integer and  $x$ be an integer with $0.454n<x\leq 0.45537n$. The number of pairings $F \in \mathcal{G}'_{16}(n)$ 
such that $\pi(F)$ has a MAI set $A$ with $|A'| = x$ is at most $$q(x,n):=\sum\limits_{i=0}^{\frac{n}{2} - x} {n \choose {\frac{n}{2} - i}}
 \cdot \frac{(\frac{n}{2} - i)! \cdot 3^{(n-2x-2i)}}{(2x+i-\frac{n}{2})! 
\cdot 2^{\frac{n}{2}-x-i} \cdot (\frac{n}{2}-x-i)!}$$
$$ \cdot \frac{(\frac{n}{2} + i)! \cdot 3^{n-2x+4i}}{(2x-3i-\frac{n}{2})! \cdot 2^{\frac{n}{2}-x+2i} \cdot
 (\frac{n}{2}-x+2i)!}$$
$$\cdot \sum\limits_{j = 0}^{n-2i-2x} {n - 2i - 2x \choose j} \cdot {n - 2x + 4i \choose j}\cdot 2^{2j} \cdot j! \cdot \left(\frac{1}{1.618}\right)^{j}$$
$$\cdot \frac{(3(2x-\frac{n}{2} - 3i))! \cdot (3(2x - \frac{n}{2} + i))!}{(3(2x-\frac{n}{2}-3i)-2(n-2i-2x)+j)!}.$$
\end{lemma}

\noindent \noindent  {\em Proof.}
By~\eqref{j251}, it is enough to show that $\sigma(n,x,16)\leq q(x,n)$.
Below we describe a procedure of constructing for every AI-pair $(A,R)$ on $[n]$ with $\alpha(A,R)=x$
 all pairings in $F \in \mathcal{G}_{16}'(n)$ for which $A$ is a MAI set. Not every obtained pairing will be in $\mathcal{G}_{16}'(n)$ and some pairings will have independence number larger than $x$, but every $F \in \mathcal{G}_{16}'(n)$ such that $A$ is a MAI set in $\pi(F)$ will be a result of this procedure.

\begin{itemize}

\item[0.] Choose nonnegative integers $n,x,i,j$ such that $n$ is even, $0.454n<x\leq 0.45537n$, $i\leq \frac{n}{2} - x$, and
$j\leq \frac{n}{2} - x-i$.

\item[1.] Choose a  set $A\subset [n]$ with $|A|=\frac{n}{2} - i$. There are $n \choose {\frac{n}{2} - i}$ ways to do it.

\item[2.] Choose a matching $R$ on $A$ with $|R|=\frac{n}{2}-x-i$.
There are  $$\frac{(\frac{n}{2} - i)!  }{(2x+i-\frac{n}{2})! \cdot 2^{\frac{n}{2}-x-i} \cdot (\frac{n}{2}-x-i)!}$$ ways to do it.
 Then there are $3^{n-2x-2i}$ ways to decide which point of each chosen end of an edge in $R$  will be the end of the corresponding edge in $F$. 

\item[3.] Similarly to Step 2, we have $$\frac{(\frac{n}{2} + i)! }{(2x-3i-\frac{n}{2})! \cdot 2^{\frac{n}{2}-x+2i} \cdot (\frac{n}{2}-x+2i)!}$$ ways to construct a matching $R'$ of $\frac{n}{2}-x+2i$ edges on $B:=[n]-A$, since
 $|B|=\frac{n}{2}+i$. After that there are $3^{n-2x+4i}$ ways to decide which point of each chosen 
 end of an edge in $R'$  will be the end of the corresponding edge in $F$. 

\item[4.] Let $Y$ (respectively, $Z$) be the set of vertices covered by the matching $R$ (respectively, $R'$).
By Lemma~\ref{YZ}, if $A$ is a MAI-set in $\pi(F)$, then the set of edges connecting $Y$ with $Z$ is a matching.
If this matching, say $M$ has $j$ edges, then there are 
  $n - 2i - 2x \choose j$ ways to choose the set of the ends of $M$ in  $Y$ and ${n - 2x + 4i \choose j}j!$ 
  ways to choose  the ends of $M$ in  $Z$. 
  Since there are $2$ free points left for each vertex in $Y$ and $Z$, we have $2^{2j}$ ways to choose which point of each vertex in $Y$ and $Z$ to be used to form an edge in $M$. 

\item[5.]
  By Lemma~\ref{J1} 
  each pairing $F \in \mathcal{G}_{16}'(n)$ containing a MAI set $A$ with $j$ edges between $Y$ and $Z$ contains at least $I(2,8)^{j}$ distinct
MAI sets of the same cardinality.  
  By Lemma~\ref{GR},
 $I(2,8)^{j}  \ge1.618^{j}$. 
 Hence by~\eqref{weight},  $\omega_x(F)\leq 1.618^{-j}$.

\item[6.]
Now we choose for each remaining free point $p$ from vertices in $Y$ a free point $q$ in a vertex in $B-Z$
and add edge $pq$. There are $$\frac{(3(2x-\frac{n}{2} - 3i))! }{(3(2x-\frac{n}{2}-3i)-2(n-2i-2x)+j)!}$$ ways to do it.

\item[7.]
 Similarly to Step 6, we choose for each remaining free point $q$ from vertices in $Z$ a free point $p$ in
 a vertex in $A-Y$ and  add edge $pq$.
 There are $$\frac{3(2x - \frac{n}{2} + i))!}{(3(2x-\frac{n}{2}+i)-2(n-2x+4i)+j)!}$$ ways to do it. 

\item[8.]
Finally, there are $3(2x-\frac{n}{2}+i)-2(n-2x+4i)+j = 10x - \frac{7n}{2} - 5i + j$ free points left in $A$ and 
$10x - \frac{7n}{2} - 5i + j$ free points left in $B$. We have $(10x - \frac{7n}{2} - 5i + j)!$ 
ways to complete a pairing on $W_n$.\qed
\end{itemize}

\medskip

In the proofs below we will use Stirling's formula: { For every $n\geq 1$,}
\begin{equation}\label{SF}
\sqrt{2\pi n}\left(\frac{n}{e}\right)^n\leq n!\leq \sqrt{2\pi n}\left(\frac{n}{e}\right)^n\,e^{1/12n}.
\end{equation}
We will also use the notation $\frac{\partial}{\partial j}$ to denote the partial derivative with respect to $j$. Moreover, we use the domain $x \ge 0$ and define $\ln(0) = -\infty$ when we consider $\ln x$.

\begin{lemma}\label{cl11}
Let $n$ be a positive even integer and $x$ be an integer satisfying  $0.454n<x\leq 0.45537n$. Let
\begin{equation}\label{j253}
\Omega=\{(\chi,\zeta,\xi): 0.454<\chi \leq 0.45537,\; 0\leq \zeta\leq  \frac{1}{2} - \chi,\; 0\leq \xi\leq 1-2\chi-2\zeta\}.
\end{equation}
Let $$f(\chi,\zeta):=$$
$$\frac{3^{\frac{1}{2}-4\chi+2\zeta} \cdot (1-2\chi-2\zeta)^{1-2\chi-2\zeta} \cdot
 (1-2\chi+4\zeta)^{1-2\chi+4\zeta} \cdot (6\chi-\frac{3}{2}+3\zeta)^{6\chi-\frac{3}{2}+3\zeta}
  \cdot (6\chi-\frac{3}{2}-9\zeta)^{6\chi-\frac{3}{2}-9\zeta}}{(2\chi+\zeta-
  \frac{1}{2})^{2\chi+\zeta-\frac{1}{2}} \cdot 2^{1-2\chi+\zeta} \cdot (\frac{1}{2}-\chi-\zeta)^{\frac{1}{2}-\chi-
  \zeta}\cdot (\frac{1}{2}-\chi+2\zeta)^{\frac{1}{2}-\chi+2\zeta} \cdot (2\chi-3\zeta-
  \frac{1}{2})^{2\chi-3\zeta-\frac{1}{2}}},$$
$$g(\chi,\zeta,\xi) := $$
$$\frac{ 2^{2\xi} \cdot (\frac{1}{1.618})^\xi }{ \xi^\xi \cdot (1-2\chi-2\zeta-\xi)^{1-2\chi-2\zeta-\xi}
 \cdot (1-2\chi+4\zeta-\xi)^{1-2\chi+4\zeta-\xi} \cdot (-\frac{7}{2}+10\chi-5\zeta+\xi)^{-\frac{7}{2}+10\chi-5\zeta+\xi}},$$
and $$h(\chi,\zeta,\xi):=f(\chi,\zeta) \cdot g(x,\zeta,\xi).$$
Then 
\begin{equation}\label{rxnij}
\frac{q(x,n)}{(3n-1)!!} = O(n^{6}) \cdot \max\{ (h(\chi,\zeta,\xi))^n\,:\, (\chi,\zeta,\xi)\in \Omega\}.
\end{equation}
\end{lemma}

\begin{proof}
We write $q(x,n)$ as a double sum of $i$ and $j$ and let $r(x,n,i,j)$ 
be the function inside the double sum of $q(x,n)$, i.e., $$q(x,n) = \sum_{i = 0}^{\frac{n}{2} - x} \sum_{j = 0}^{n-2x-2i} r(x,n,i,j).$$
Then certainly, 
$$q(x,n)\leq n^2 \cdot \max\{r(x,n,i,j): 0\leq i\leq \frac{n}{2}-x,\, 0\leq j\leq n-2x-2i\}.$$

So, it is enough to estimate $r(x,n,i,j)$. We know that 
$$r(x,n,i,j)=\frac{n!}{(\frac{n}{2}-i)! \cdot (\frac{n}{2}+i)!} \cdot \frac{(\frac{n}{2}-i)! \cdot 3^{n-2x-2i}}{(2x+i-\frac{n}{2})! 
\cdot 2^{\frac{n}{2}-x-i} \cdot (\frac{n}{2}-x-i)!}$$
$$\cdot \frac{(\frac{n}{2}+i)! \cdot 3^{n-2x+4i}}{(2x-3i-\frac{n}{2})! \cdot 2^{\frac{n}{2}-x+2i}
 \cdot (\frac{n}{2}-x+2i)!}\cdot \frac{(n-2i-2x)!}{j! \cdot (n-2i-2x-j)!} $$
$$ \cdot \frac{(n-2x+4i)!}{j! \cdot (n-2x+4i-j)!} \cdot 2^{2j} \cdot j! \cdot (\frac{1}{1.618})^{j}
\cdot \frac{(6x-\frac{3n}{2}-9i)! \cdot (6x-\frac{3n}{2}+3i)!}{(10x-\frac{7n}{2}-5i+j)!}.$$
Recall that $$(3n-1)!! \ge \frac{(3n)!!}{3n} \ge \frac{\sqrt{(3n)!}}{3n}.$$
Therefore,
$$\frac{r(x,n,i,j)}{(3n-1)!!} \le \frac{n! \cdot (3n)}{((3n)!)^{\frac{1}{2}}} \cdot \frac{3^{n-2x-2i}}{(2x+i-\frac{n}{2})! 
\cdot 2^{\frac{n}{2}-x-i} \cdot (\frac{n}{2}-x-i)!} $$

$$\cdot \frac{ 3^{n-2x+4i}}{(2x-3i-\frac{n}{2})! \cdot 2^{\frac{n}{2}-x+2i} \cdot (\frac{n}{2}-x+2i)!}\cdot \frac{(n-2i-2x)!}{j! \cdot (n-2i-2x-j)!}$$

$$ \cdot \frac{(n-2x+4i)!}{(n-2x+4i-j)!} \cdot 2^{2j} \cdot
 (\frac{1}{1.618})^{j}\cdot \frac{(6x-\frac{3n}{2}-9i)! \cdot (6x-\frac{3n}{2}+3i)!}{(10x-\frac{7n}{2}-5i+j)!}.$$
Introducing new variables  $\chi:=\frac{x}{n}$ , $\zeta:=\frac{i}{n}$, and $\xi:=\frac{j}{n}$ and  using Stirling's formula \eqref{SF}, we get
$$\frac{r(x,n,i,j)}{(3n-1)!!} = O(n^4) \cdot \frac{\left(  \frac{n}{e} \right)^{n} \cdot \left(  \frac{n}{e} \right)^{(1-2\zeta-2\chi)n} \cdot \left(  \frac{n}{e} \right)^{(1-2\chi+4\zeta)n} \cdot \left(  \frac{n}{e} \right)^{(6\chi-\frac{3}{2}-9\zeta)n} \cdot \left(  \frac{n}{e} \right)^{(6\chi - \frac{3}{2} + 3\zeta)n}}{\left(  \frac{n}{e} \right)^{\frac{3}{2}n} \cdot \left(  \frac{n}{e} \right)^{(2\chi+\zeta-\frac{1}{2})n} \cdot \left(  \frac{n}{e} \right)^{(\frac{1}{2}-\chi-\zeta)n} \cdot \left(  \frac{n}{e} \right)^{(2\chi-3\zeta-\frac{1}{2})n} \cdot \left(  \frac{n}{e} \right)^{(\frac{1}{2} - \chi + 2\zeta)n} \cdot \left(  \frac{n}{e} \right)^{\xi n}}$$

$$\cdot \frac{1}{\left(  \frac{n}{e} \right)^{(1-2\zeta-2\chi-\xi)n} \cdot \left(  \frac{n}{e} \right)^{(1-2\chi+4\zeta-\xi)n} \cdot \left(  \frac{n}{e} \right)^{(10\chi-\frac{7}{2}-5\zeta+\xi)n}} \cdot \left( f(\chi,\zeta) \cdot g(\chi,\zeta,\xi) \right)^{n}.$$
 Therefore, 
$$\frac{r(x,n,i,j)}{(3n-1)!!} = O(n^{4}) \cdot (h(\chi,\zeta,\xi))^n.$$
This proves the lemma.
\end{proof}

Recall that the domain of $h(\chi,\zeta,\xi)$ is $\Omega$ defined in~\eqref{j253}.
Our main goal now is to show that 
\begin{equation}\label{czx}
\max\limits_{(\chi,\zeta,\xi) \in \Omega}{h(\chi,\zeta,\xi)} \leq 0.999983 < 1.
\end{equation}
We do this in the next section, and then Theorem~\ref{solve} easily follows.

\section{Proof of~\eqref{czx}}

In order to find the maximum value of $h(\chi,\zeta,\xi)$ for a fixed $\chi$, we will maximize $\ln(h(\chi,\zeta,\xi))$.
We first find the value of $\xi$ in terms of $\chi$ and $\zeta$ that maximizes $\ln(g(\chi,\zeta,\xi))$. By definition,
$$\ln(g(\chi,\zeta,\xi)) =  \xi \ln(\frac{4}{1.618}) - (\xi\ln(\xi)+(1-2\zeta-2\chi-\xi)\ln(1-2\zeta-2\chi-\xi)$$
$$+(1-2\chi+4\zeta-\xi)\ln(1-2\chi+4\zeta-\xi)+(10\chi-\frac{7}{2}-5\zeta+\xi)\ln(10\chi-\frac{7}{2}-5\zeta+\xi)).$$
Hence
$$\frac{\partial \ln(g(\chi,\zeta,\xi))}{\partial \xi} = \ln(1 - 2\chi - 2\zeta -\xi) + \ln(1 - 2\chi +4\zeta -\xi) - \ln(10\chi-5\zeta+\xi-\frac{7}{2}) - \ln(\xi) + \ln(\frac{4}{1.618})$$
 $$= \ln\left(  \frac{(1-2\chi-2\zeta-\xi) \cdot (1-2\chi+4\zeta-\xi) \cdot \frac{4}{1.618}}{ \xi \cdot (10\chi-5\zeta+\xi-\frac{7}{2})} \right).$$
In order to solve $$\frac{\partial \ln(g(\chi,\zeta,\xi))}{\partial \xi} = 0,$$ we solve the equivalent equation
$$p(\xi):=4 \cdot (1-2\chi-2\zeta-\xi) \cdot (1-2\chi+4\zeta-\xi)  - 1.618 \cdot \xi \cdot (10\chi-5\zeta+\xi-\frac{7}{2}) = 0,$$
where $p(\xi)$ has domain $0 \le \xi  \le 1 -2\chi -2\zeta.$
By the quadratic formula, the roots are
 $$\xi_1 = \frac{-b-\sqrt{b^2-4ac}}{2a}\quad
\mbox{and}\quad
\xi_2 = \frac{-b+\sqrt{b^2-4ac}}{2a},$$
where  $$a = 2.382,$$
\begin{equation}\label{b=}
 b = -0.18 \chi + 0.09 \zeta - 2.337,
\end{equation}
\begin{equation}\label{c=}
c = 16 \chi^2 - 32 \zeta^2 - 16 \chi + 8 \zeta - 16 \chi \zeta + 4.
\end{equation}
Moreover, for fixed $\chi$ and $\zeta$ satisfying $0.454 \le \chi \le 0.45537$ and $\chi + \zeta \le \frac{1}{2},$ $p(\xi)$ 
is a parabola opening upward with $\xi_1 \le 1-2\chi-2\zeta \le \xi_2$ because $p(1-2\chi-2\zeta) \le 0$, and $g(\chi,\zeta,\xi)$ is a continuous function on $\xi$. 
Therefore, the maximum of $g(\chi,\zeta,\xi)$ can only be attained at $\xi = \xi_1$.

Let 
$g_1(\chi,\zeta) = g(\chi,\zeta,\xi_1(\chi,\zeta)).$ 
For each fixed $\chi$,
 consider the maximum of $$h_1(\chi,\zeta):=f(\chi,\zeta) \cdot g_1(\chi,\zeta).$$
By definition,
$$\ln(h_1) = (\frac{1}{2} - 4\chi + 2\zeta)\ln(3) + (1-2\chi-2\zeta)\ln(1-2\chi-2\zeta) + (1-2\chi+4\zeta)\ln(1-2\chi+4\zeta) $$
$$ +(6\chi-\frac{3}{2}+ 3\zeta)\ln(6\chi-\frac{3}{2}+3\zeta)+ (6\chi - \frac{3}{2}-9\zeta)\ln(6\chi-\frac{3}{2}-9\zeta)+ \xi_1(\chi,\zeta) \cdot \ln(\frac{4}{1.618})$$
$$-(2\chi+\zeta-\frac{1}{2})\ln(2\chi+\zeta-\frac{1}{2}) - (1-2\chi+\zeta)\ln(2) - (\frac{1}{2}-\chi-\zeta)\ln(\frac{1}{2}- \chi-\zeta) - (\frac{1}{2}- \chi+2\zeta)\ln(\frac{1}{2}-\chi+2\zeta)$$
$$-(2x-3\zeta-
\frac{1}{2})\ln(2\chi - 3\zeta - \frac{1}{2}) - \xi_1(\chi,\zeta) \cdot \ln(\xi_1(\chi,\zeta)) - (1-2\zeta-2\chi-\xi_1(\chi,\zeta))\ln(1-2\zeta-2\chi-\xi_1(\chi,\zeta))$$
$$-(1-2\chi+4\zeta-\xi_1(\chi,\zeta))\ln(1-2\chi+4\zeta-\xi_1(\chi,\zeta)) - (10\chi-\frac{7}{2}-5\zeta+\xi_1(\chi,\zeta))\ln(10\chi-\frac{7}{2}-5\zeta+\xi_1(\chi,\zeta)),$$
and
$$\frac{\partial \ln(h_1)}{\partial \zeta} =-4\ln(3)+\ln(2)-3+2\ln(2\chi+\zeta-\frac{1}{2}) - \ln(\frac{1}{2}-\chi-\zeta) + 2\ln(\frac{1}{2} - \chi +2\zeta) - 6\ln(2\chi-3\zeta-\frac{1}{2}) $$
$$+ \ln(\frac{4}{1.618}) \cdot \frac{\partial \xi_1(\chi,\zeta)}{\partial \zeta} - 
\frac{\partial \xi_1(\chi,\zeta)}{\partial \zeta} \cdot (\ln(\xi_1(\chi,\zeta)) + 1) + (2 + \frac{\partial \xi_1(\chi,\zeta)}{\partial \zeta}) \cdot (\ln(1-2\zeta-2\chi-\xi_1(\chi,\zeta))+1)$$
$$ +( \frac{\partial \xi_1(\chi,\zeta)}{\partial \zeta}-4) \cdot (\ln(1-2\chi+4\zeta-\xi_1(\chi,\zeta))+1) + 
(5-\frac{\partial \xi_1(\chi,\zeta)}{\partial \zeta}) \cdot (\ln(10\chi-\frac{7}{2}-5\zeta+\xi_1(\chi,\zeta))+1),$$
where
$$\frac{\partial \xi_1(\chi,\zeta)}{\partial \zeta} = \frac{1}{2a} \cdot (-\frac{\partial b}{\partial \zeta} - 
\frac{1}{2} \cdot (b^2 - 4ac)^{-\frac{1}{2}} \cdot (2b\frac{\partial b}{\partial \zeta}- 4a \frac{\partial c}{\partial \zeta})),$$
$$\frac{\partial b}{\partial \zeta} = 0.09,$$
$$\frac{\partial c}{\partial \zeta} = -16\chi - 64 \zeta + 8.$$



\begin{lemma}\label{max}
When $\chi=0.454,$ the maximum of 
$h(\chi,\zeta,\xi)$ over $0 \le \zeta \le 0.046$ and $0\leq \xi\leq 0.092-2\zeta$ is at most $0.999983$.
\end{lemma}

{\bf Proof} Fix  $\chi=0.454$. For $0 \le \zeta \le 0.046,$ denote $$\xi_1'(\zeta):=\frac{\partial \xi_1(0.454,\zeta)}{\partial \zeta} \text{, }
 \xi_1''(\zeta):=\frac{\partial^2 \xi_1(0.454,\zeta)}{\partial \zeta^2} \text{, and } \xi_1(\zeta):=\xi_1(0.454,\zeta).$$
We have
$$\frac{\partial\ln(h_1(0.454,\zeta))}{\partial \zeta} = -4\ln(3) + \ln(2) - 3 +2\ln(0.408+\zeta) - \ln(0.046-\zeta) + 2\ln(0.046+2\zeta)  $$
$$ - 6\ln(0.408-3\zeta)+\ln(\frac{4}{1.618}) \cdot \xi_1'(\zeta) - 
\xi_1'(\zeta) \cdot (\ln(\xi_1(\zeta)) + 1) + (2 + \xi_1'(\zeta)) \cdot (\ln(0.092-2\zeta-\xi_1(\zeta))+1)  $$
$$+ ( \xi_1'(\zeta)-4) \cdot (\ln(0.092+4\zeta-\xi_1(\zeta))+1) + (5-\xi_1'(\zeta)) \cdot (\ln(1.04-5\zeta+\xi_1(\zeta))+1),$$
$$\frac{\partial^2 \ln(h_1(0.454,\zeta))}{\partial \zeta^2} = \frac{1}{0.046 - \zeta} + \frac{4}{0.046+2\zeta} + \frac{2}{0.408+\zeta} + 
\frac{18}{0.408-3\zeta} $$
$$ + \ln(\frac{4}{1.618}) \cdot \xi_1''(\zeta)- \xi_1''(\zeta) \cdot (\ln(\xi_1(\zeta))+1) - (\xi_1'(\zeta))^2 \cdot \frac{1}{\xi_1(\zeta)} $$
$$+\xi_1''(\zeta) \cdot (\ln(0.092-2\zeta-\xi_1(\zeta))+1) - (2 + \xi_1'(\zeta))^2 \cdot \frac{1}{0.092-2\zeta-\xi_1(\zeta)} $$
$$+ \xi_1''(\zeta) \cdot (\ln(0.092+4\zeta-\xi_1(\zeta))+1) - (\xi_1'(\zeta)-4)^2 \cdot \frac{1}{0.092+4\zeta-\xi_1(\zeta)} $$
$$- \xi_1''(\zeta) \cdot (\ln(1.04-5\zeta+\xi_1(\zeta))+1)- (\xi_1'(\zeta)-5)^2 \cdot \frac{1}{1.04-5\zeta+\xi_1(\zeta)}, $$
where
\begin{equation}\label{j30}
\xi_1''(\zeta) = \frac{1}{2a} \cdot (b^2 - 4ac)^{-\frac{1}{2}}\cdot \left(        \frac{1}{4} \cdot (b^2 - 4ac)^{-1} \cdot  
 (2b\frac{\partial b}{\partial \zeta}- 4a \frac{\partial c}{\partial \zeta})^2  
- \frac{1}{2} \cdot (2 (\frac{\partial b}{\partial \zeta})^2 + 1024a ) \right).
\end{equation}

We will show that
\begin{equation}\label{j29}
 \frac{\partial^2 \ln(h_1(0.454,\zeta))}{\partial \zeta^2} < 0, \text{ for all } 0 \le \zeta < 0.046.
\end{equation}
This will guarantee that if we find a solution $\zeta_0\in [0,0.046)$ of the equation
$\frac{\partial \ln(h_1(0.454,\zeta))}{\partial \zeta}=0,$
then the maximum   of $ h_1(0.454,\zeta)$ over $\zeta\in [0,0.046)$ is attained at $\zeta_0$.

\begin{claim}\label{ppj} For each $\zeta\in [0,0.046)$,
$-27.336 \le \xi_1''(\zeta) < -24.822.$ 
\end{claim}

\begin{proof}
By~\eqref{b=} and~\eqref{c=}, for $\chi=0.454$, the function $\Delta(\zeta) := b^2 - 4ac$ is quadratic in $\zeta$ with derivative 
$$\Delta'(\zeta) = 2b\frac{\partial b}{\partial \zeta}- 4a 
\frac{\partial c}{\partial \zeta},$$
which is linear in $\zeta$ and has minimum at $\zeta = 0$ and maximum at $\zeta = 0.046.$ Therefore,  $$-7.45 \le \Delta'(0) \le 
\Delta'(\zeta) \le \Delta'(0.046) \le 20.61$$ for each $\zeta\in [0,0.046)$. Also for such $\zeta$, $$ \Delta''(\zeta) = 2(\frac{\partial b}{\partial \zeta})^2-
 4a \frac{\partial^2 c}{\partial \zeta^2} = 2 \cdot 0.09^2 - 4 \cdot 2.382 \cdot (-64) \in (609.8, 609.81),$$ so $\Delta(\zeta)$ is a parabola
 opening upward with minimum attained at the unique root $\zeta_\xi$ of the equation $\Delta'(\zeta)=0.$ By $\Delta'(0.012213)<-0.00039,$ $\Delta'(0.012214)>0.000219,$ and the above statements,$$0.012213 \le \zeta_\xi \le 0.012214.$$
Hence $\Delta(\zeta_\xi)$ satisfies
$$5.4821 \le 5.48214 -  0.0004 \cdot0.000001\le \Delta(0.012213) + \Delta'(0.012213) \cdot 0.000001 $$
$$\le \Delta(\zeta_\xi) \le \Delta(0.012213) \le5.4822,$$
and the maximum of $\Delta(\zeta)$ over $\zeta\in [0,0.046)$ is attained at $\zeta=0.046$
and satisfies
$$5.83019 \le \Delta(0.046) \le 5.8302.$$
Therefore,  for each $\zeta\in [0,0.046)$,
$$0.41415 \le \frac{1}{\sqrt{5.8302}} \le (\Delta(\zeta))^{-\frac{1}{2}} \le \frac{1}{\sqrt{5.4821}} \le 0.427098,$$
$$0.17152 \le (\Delta(\zeta))^{-1} \le 0.182412.$$
Thus by~\eqref{j30},
\begin{equation}\label{lwpdp}
-27.336 \le \frac{1}{2 \cdot 2.382} \cdot 0.427098 \cdot (0 - 0.5 \cdot 609.81 ) \le \xi_1''(\zeta)
\end{equation}
$$  = \frac{1}{2a} \cdot (\Delta(\zeta))^{-\frac{1}{2}}\cdot \left(        \frac{1}{4} \cdot (\Delta(\zeta))^{-1} \cdot   (\Delta'(\zeta))^2  
- \frac{1}{2} \cdot \Delta''(\zeta) \right)$$
\begin{equation}\label{uppdp}
\le \frac{1}{2 \cdot 2.382} \cdot 0.41415 \cdot (0.25 \cdot 0.182412 \cdot (20.61)^2 - 0.5 \cdot 609.8 ) \le -24.822.
\end{equation}
This proves Claim~\ref{ppj}.
\end{proof}

\begin{claim}\label{jprime} For each $\zeta\in [0,0.046)$,
$ -0.91445 \le \xi_1'(0.046) \le \xi_1'(\zeta) \le \xi_1'(0) \le 0.31359$.
\end{claim}

\begin{proof}
By Claim~\ref{ppj}, $\xi_1'(\zeta)$ is a decreasing function on $0 \le \zeta \le 0.046.$
\end{proof}


To prove~\eqref{j29}, we write  $\frac{\partial^2 \ln(h_1(0.454,\zeta))}{\partial \zeta^2}$ in a form
\begin{equation}\label{A1-5}
 \frac{\partial^2 \ln(h_1(0.454,\zeta))}{\partial \zeta^2}=A_1(\zeta)+A_2(\zeta)+A_3(\zeta)+A_4(\zeta)+A_5(\zeta),
\end{equation}
 and then  bound  these expressions separately so that
the sum of the upper bounds will be negative for each $\zeta\in [0,0.046)$. By definition
$$\frac{\partial^2 \ln(h_1)}{\partial \zeta^2} = \frac{1}{0.046 - \zeta} - \xi_1''(\zeta) \cdot \ln(\xi_1(\zeta))  - 
 \frac{(\xi_1'(\zeta))^2}{\xi_1(\zeta)}+ \frac{4}{0.046+2\zeta} + \frac{2}{0.408+\zeta} $$
$$+ \frac{18}{0.408-3\zeta} +\ln(\frac{4}{1.618}) \cdot \xi_1''(\zeta) +\xi_1''(\zeta) \cdot \ln(0.092-2\zeta-\xi_1(\zeta)) -
 \frac{ (2 + \xi_1'(\zeta))^2}{0.092-2\zeta-\xi_1(\zeta)} $$
$$+\xi_1''(\zeta) \cdot \ln(0.092+4\zeta-\xi_1(\zeta)) -  \frac{(\xi_1'(\zeta)-4)^2}{0.092+4\zeta-\xi_1(\zeta)} $$
$$- \xi_1''(\zeta) \cdot \ln(1.04-5\zeta+\xi_1(\zeta))-  \frac{(\xi_1'(\zeta)-5)^2}{1.04-5\zeta+\xi_1(\zeta)}. $$

Let
\begin{equation}\label{A1}
A_1(\zeta) := \frac{1}{0.046 - \zeta} - (\xi_1'(\zeta))^2 \cdot \frac{1}{\xi_1(\zeta)} -  \frac{(2+\xi_1'(\zeta))^2}{0.092 - 2\zeta - \xi_1(\zeta)},
\end{equation}
\begin{equation}\label{A2}
A_2(\zeta) := \xi_1''(\zeta) \cdot \ln(0.092 - 2\zeta - \xi_1(\zeta)) - \xi_1''(\zeta) \cdot \ln(\xi_1(\zeta)),
\end{equation}
\begin{equation}\label{A3}
A_3(\zeta):=\ln(\frac{4}{1.618}) \cdot \xi_1''(\zeta),
\end{equation}
\begin{equation}\label{A4}
A_4(\zeta) := \frac{4}{0.046+2\zeta} + \frac{2}{0.408+\zeta} + \frac{18}{0.408 - 3\zeta},\;\mbox{and}
\end{equation}
\begin{equation}\label{A5}
A_5(\zeta):=\xi_1''(\zeta) \cdot \ln(0.092+4\zeta-\xi_1(\zeta)) -  \frac{(\xi_1'(\zeta)-4)^2}{0.092+4\zeta-\xi_1(\zeta)}
\end{equation}
$$- \xi_1''(\zeta) \cdot \ln(1.04-5\zeta+\xi_1(\zeta))-  \frac{(\xi_1'(\zeta)-5)^2}{1.04-5\zeta+\xi_1(\zeta)},$$
so that~\eqref{A1-5} holds.

\begin{claim}\label{c31} For each $\zeta\in [0,0.046)$,
$A_1(\zeta) < 0.$
\end{claim}

\noindent  {\em Proof.}
Since $0.092 - 2\zeta - \xi_1(\zeta) \ge 0$ and $\xi_1(\zeta) \ge 0,$ by Claim~\ref{jprime},
$$A_1(\zeta) = \frac{1}{0.046 - \zeta} - (\xi_1'(\zeta))^2 \cdot \frac{1}{\xi_1(\zeta)} - (2+\xi_1'(\zeta))^2 \cdot \frac{1}{0.092 - 2\zeta - \xi_1(\zeta)} $$
$$\le \frac{1}{0.046 - \zeta} - (\xi_1'(\zeta))^2 \cdot \frac{1}{0.092-2\zeta} - (2+\xi_1'(\zeta))^2 \cdot \frac{1}{0.092 - 2\zeta}$$
$$=\frac{1}{0.046 - \zeta} - \frac{(\xi_1'(\zeta)+1)^2+1}{0.046 - \zeta} = -\frac{(\xi_1'(\zeta)+1)^2}{0.046 - \zeta} < 0.
\qed $$

\begin{claim}\label{c32} For each $\zeta\in [0,0.046)$,
$A_2(\zeta) < 0.$
\end{claim}

\noindent  {\em Proof.} Let $\zeta\in [0,0.046)$. By Claim~\ref{ppj},
inequality $A_2(\zeta)<0$ is equivalent to $$0.092 - 2\zeta - \xi_1(\zeta) > \xi_1(\zeta).$$ 
Let $y(\zeta) = 0.092 - 2\zeta -2\xi_1(\zeta)$. By Claim~\ref{jprime},
 $$y'(\zeta) = -2 - 2\xi_1'(\zeta) < 0.$$
Therefore, $y(\zeta) > y(0.046) = 0$ for each $\zeta\in [0,0.046)$. This proves the claim.\qed



\begin{claim}\label{c33} For each $\zeta\in [0,0.046)$,
$A_3(\zeta) \le -22.46.$
\end{claim}

\noindent  {\em Proof.} This follows from the definition~\eqref{A3}, since  $\xi_1''(\zeta) \le -24.822$ by Claim~\ref{ppj}.\qed

\begin{claim}\label{c34} 
The function
 $A_4'(\zeta)$ has exactly one root $d_\zeta$ in the interval $[0,0.046]$. Furthermore, $d_\zeta\in (0.0355167,0.0355168)$, and   $A_4(\zeta)$ is decreasing on $[0,d_\zeta]$ 
and increasing on $[d_\zeta, 0.046].$
\end{claim}

\noindent  {\em Proof.} By Definition~\eqref{A4},
$$A_4'(\zeta) = -\frac{2}{(\zeta+0.023)^2} - \frac{2}{(\zeta+0.408)^2} + \frac{6}{(\zeta - 0.136)^2}$$
and
$$A_4''(\zeta) = \frac{4}{(\zeta+0.023)^3} + \frac{4}{(\zeta + 0.408)^3} - \frac{12}{(\zeta - 0.136)^3}.$$
The last expression is positive for all $\zeta\in [0,0.046]$, so function
 $A_4'(\zeta)$ may have at most one root on $[0,0.046]$. On the other hand,
$A_4'(0.0355167) < - 0.002$ and  $A_4'(0.0355168) > 0.0006$. This proves the claim. 
\qed




\begin{claim}\label{c35} For each $\zeta\in [0,0.046)$,
$A_4(\zeta)+A_5(\zeta) \le 20.$
\end{claim}

\noindent  {\em Proof.} Let 
$$z_1(\zeta)=0.092 + 4\zeta - \xi_1(\zeta)\quad\mbox{and}\quad z_2(\zeta)=1.04 - 5\zeta + \xi_1(\zeta).$$
By Claim~\ref{jprime},  $z'_1(\zeta) = 4 - \xi_1'(\zeta) > 0$  and $z'_2(\zeta) = -5 + \xi_1'(\zeta) < 0$  
for each $\zeta\in [0,0.046)$. So,
\begin{equation}\label{ju3}
\mbox{\em  $z_1(\zeta)$ is increasing and $z_2(\zeta)$ is decreasing on $ [0,0.046)$.}
\end{equation}
 Since 
 $$ z_1(\zeta)<z_1(0.046)<z_2(0.046)<z_2(\zeta)$$
 for each $\zeta\in [0,0.046)$, Definitions~\eqref{A4} and~\eqref{A5} together with  Claim~\ref{ppj} yield
$$A_4(\zeta)+A_5(\zeta)=A_4(\zeta)+\xi_1''(\zeta) \cdot \ln(z_1(\zeta)) -   \frac{(\xi_1'(\zeta)-4)^2}{z_1(\zeta)}
 -\xi_1''(\zeta) \cdot \ln(z_2(\zeta))-  \frac{(\xi_1'(\zeta)-5)^2}{z_2(\zeta)} $$
$$\le A_4(\zeta)-27.336 \cdot (\ln(z_1(\zeta)) - \ln(z_2(\zeta))) - \frac{ (\xi_1'(\zeta)-4)^2}{z_1(\zeta)}- \frac{ (\xi_1'(\zeta)-5)^2}{z_2(\zeta)}=: Q(\zeta).$$

Since  $\zeta\in [0,0.046)$, it belongs to the interval $[0.001k,0.001(k+1))$ for some integer $0\leq k\leq 45$.
We consider $3$ cases.

\smallskip

{\em Case 1:}  $0\leq k \leq 34$.  Then by Claim \ref{c34} and~\eqref{ju3}, for each
$\zeta\in [0.001k,0.001(k+1))$,
$$A_4(0.001k) \ge A_4(\zeta),$$
$$z_1(\zeta) \ge z_1(0.001k)\text{, and }z_2(0.001k) \ge z_2(\zeta).$$
Therefore,
$$Q(\zeta) \le  M_1(k):=A_4( 0.001k) -27.336 \cdot (\ln(z_1(0.001k)) -  \ln(z_2(0.001k)))$$
$$ - \frac{(\xi_1'( 0.001k)-4)^2}{z_1(0.001(k+1))}  - \frac{(\xi_1'(  0.001k)-5)^2}{z_2(0.001k)}.$$
The bounds for $M_1(k)$ certifying that $M_1(k)<20$ for each $0\leq k \leq 34$ are given in Table 2 in
Appendix~1.

{\em Case 2:} $k=35$.  Similarly to Case 1,
$$Q(\zeta) \le  \max(A_4(0.035),A_4(0.036)) -27.336 \cdot (\ln(z_1(0.035))
 -  \ln(z_2(0.035)))$$
$$ - \frac{(\xi_1'(0.035)-4)^2}{z_1(0.036)} 
 - \frac{(\xi_1'(0.035)-5)^2}{z_2(0.035)} $$
$$<98.404 - 27.336 \cdot (-1.5-(-0.135)) -94 -36.3< 5.5< 20.$$

{\em Case 3:}  $36\leq k \leq 45$. Again, similarly to Case 1,
$$Q(\zeta) \le  M_3(k):=A_4( 0.001(k+1)) -27.336 \cdot (\ln(z_1(0.001k)) -  \ln(z_2(0.001k)))$$
$$ - \frac{(\xi_1'( 0.001k)-4)^2}{z_1(0.001(k+1))}  - \frac{(\xi_1'(  0.001k)-5)^2}{z_2(0.001k)}.$$
The bounds for $M_1(k)$ certifying that $M_1(k)<20$ for each $36\leq k \leq 45$ are given in Table 1 in
Appendix~1.
\qed

Thus by~\eqref{A1-5} and Claims~\ref{c31}--\ref{c35}, for each  $\zeta\in [0,0.046)$,
  $$\frac{\partial^2 \ln(h_1(0.454,\zeta))}{\partial \zeta^2} = \sum_{i=1}^5A_i(\zeta)  < -22.46 + 20 = -2.46 < 0.$$
We also can check by plugging in the values that
 $$\frac{\partial \ln(h_1(0.454,0.0228718))}{\partial \zeta} > 7.54 \cdot 10^{-8} \text{, and } \frac{\partial \ln(h_1(0.454,0.0228719))}{\partial \zeta} < -9 \cdot 10^{-6}.$$
Thus, the derivative of $h_1(0.454,\zeta)$ equals $0$ at a unique 
 $ \zeta_1\in( 0.0228718 , 0.0228719)$.

Recall that $h_1(0.454,\zeta)>0$ for  $\zeta\in [0,0.046)$. So, after comparing the value 
$h_1(0.454,0.0228719)$
 with the boundary values $h_1(0.454,0)$ and $h_1(0.454,0.46) $, we conclude that the maximum of
 $h_1(0.454,\zeta)$ is attained at $\zeta_1$.
 We can plug in numbers into a computer and obtain that $$h_1(0.454, 0.0228718) \le 0.999982,$$ 
 $$\frac{\partial \ln(h_1(0.454, 0.0228718))}{\partial \zeta} \le 1 \cdot 10^{-7},$$ 
and
$$\frac{\partial h_1(0.454, 0.0228718)}{\partial \zeta} = h_1(0.454,0.0228718) \cdot \frac{\partial \ln(h_1(0.454, 0.0228718))}{\partial \zeta}\le 1 \cdot 10^{-7},$$
which implies that $$h_1(0.454, \zeta_1) \le h_1(0.454, 0.0228718)+ 1 \cdot 10^{-7} \cdot 0.0000001 \le 0.999983. 
\qed$$


The proof of the next lemma is similar but significantly simpler. It is mostly a routine bounding some expressions. So, we present the proof
of Lemma~\ref{ph} in Appendix~2.

\begin{lemma}\label{ph} For every 
$$(\chi,\zeta,\xi) \in \Omega=\{(\chi,\zeta,\xi): 0.454<\chi \leq 0.45537,\; 0\leq \zeta\leq  \frac{1}{2} - \chi,\; 0\leq \xi\leq 1-2\chi-2\zeta\},$$
we have
\begin{equation}\label{ju31}
\frac{\partial \ln(h(\chi,\zeta,\xi))}{\partial \chi} < 0.\qed
\end{equation}
\end{lemma}


Since $h(\chi,\zeta,\xi) > 0$ for each  $(\chi,\zeta,\xi) \in \Omega,$
 Lemma~\ref{ph} yields that for each fixed $\zeta$ and $\xi$,  the maximum of $h(\chi,\zeta,\xi)$ over
 $(\chi,\zeta,\xi) \in \Omega$
 is attained at $\chi=0.454.$  By Lemma~\ref{max},
 this maximum 
  is at most 0.999983.  This yields~\eqref{czx}.
  
\section{Completion of the proof of Theorem~\ref{solve}}
  By~\eqref{czx} and  Lemma~\ref{cl11}, for  all positive integers $n$ and $x$ such that $n$ is even and $0.454n<x\leq 0.45537n$,
 $$
\frac{q(x,n)}{(3n-1)!!} \leq O(n^{6}) \cdot 0.999983^n.
$$
It follows that
   \begin{equation}\label{rlast}
\frac{1}{(3n-1)!!}\sum_{x=\lceil 0.454n\rceil}^{\lfloor 0.45537n\rfloor}q(x,n) \leq O(n^{7}) \cdot 0.999983^n
\to 0\quad \mbox{as $n\to\infty$.}
\end{equation}
Thus by Lemma~\ref{formula},
 the number of pairings $F \in \mathcal{G}'_{16}(n)$ with
 $0.454n< \alpha(F)\leq 0.45537n$ is  $o\left((3n-1)!!\right)$.
Together with Theorem~\ref{MK1}, this means that almost no pairings have independence ratio larger than 0.454. Thus by Corollary~\ref{MK2} we conclude 
that almost no $n$-vertex $3$-regular graphs of girth at least $16$ have independence ratio larger than $0.454$.
This proves Theorem~\ref{solve} and thus also Theorem~\ref{main theorem}.






\bigskip\noindent
{\bf Acknowledgment.} We thank Jan Volec for helpful discussion and bringing~\cite{C1} to our attention. We thank a referee for the valuable comments.

\newpage

{\bf \Large Appendix 1: Tables for Claim~\ref{c35}}
\vspace{1cm}

\begin{tabular}{c|c|c|c|c|c|c|c}
\hline
k & $A_4(0.001k)$ & $- \ln(z_1(0.001k))$  & $ \ln(z_2(0.001k))$   & $ - \frac{(\xi_1'( 0.001k)-4)^2}{z_1(0.001(k+1))}  $  &  $ - \frac{(\xi_1'(  0.001k)-5)^2}{z_2(0.001k)} $  & $M_1(k)$     \\ \hline
 0 &   135.9762 & 2.553562 & 0.05277836& -166.7356 & -20.83335& 19.7\\ \hline
 1 &  132.6679  &  2.507105&  0.04831009 & -161.790& -21.1686&19.6  \\ \hline
 2&   129.6543 & 2.462392 &  0.04379588   & -157.2333& -21.50947&19.5 \\ \hline
 3 &   126.903 &  2.419288& 0.03923514  & -153.0194& -21.85574&19.3\\ \hline
 4&  124.384  &  2.377674 &  0.03462729 & -149.1122&  -22.20758&19.1\\ \hline
 5 &   122.0728  &  2.33745&   0.02997174 &-145.4794 &-22.56505 &18.8\\ \hline
 6 &   119.9504  &  2.298492& 0.02526786 & -142.0935& -22.92824&18.5\\ \hline
 7 &   117.9977 &  2.260743  &  0.02051507& -138.9302& -23.29722&18.2\\ \hline
 8 &    116.1989 &   2.224115 & 0.01571273 &  -135.9683& -23.67205&17.8 \\ \hline
9  &   114.5404 &  2.188538  &  0.01086022& -133.1892& -24.05282&17.5\\ \hline
 10 &   113.0099& 2.153948 & 0.005956888 & -130.5765& -24.43961&17.1\\ \hline
 11 &    111.5969&  2.120286  &  0.001002109&-128.1157 & -24.83248&16.7\\ \hline
 12 &   110.292 &   2.0876 & -0.004004782&-125.793 & -25.23152&16.3 \\ \hline
13  &   109.0867 &  2.055542  &  -0.009064451&-123.5994 & -25.63682&15.8\\ \hline
 14 &    107.9738&   2.024367  & -0.01417756& -121.5220& -26.04845&15.4\\ \hline
 15 &    106.9466&  1.993934  &  -0.01934483& -119.5525& -26.4664& 15.0\\ \hline
 16 &   106.000 &    1.964204 & -0.02456692& -117.6823&  -26.89104&14.5\\ \hline 
 17 &  105.1262  &  1.935144 &  -0.02984456& -115.9041& -27.32218&14.0\\ \hline
18  &    104.3229& 1.90673   & -0.03517846& -114.2111& -27.76000&13.6\\ \hline
  19&    103.585 &  1.878905  & -0.04056935& -112.5970& -28.20460& 13.1\\ \hline
20  &   102.9088 &  1.851668  & -0.04601797 & -111.0562& -28.65606&12.6\\ \hline
21 &    102.2906 &  1.824985  & -0.05152507 & -109.5836& -29.1144&12.1\\ \hline
  22&   101.7273 &  1.798832  &-0.05709143 & -108.1746& -29.57998&11.6\\ \hline
 23 &   101.217 & 1.773185   & -0.06271781& -106.8248& -30.05263&11.1\\ \hline
 24 &    100.7543&  1.748024   &  -0.06840502 & -105.5304&-30.53255 &10.7\\ \hline
 25 &    100.3398 & 1.723329  & -0.07415386 & -104.2877 & -31.01984&10.2\\ \hline
 26 &    99.97009 & 1.699082  &  -0.07996514& -103.0934&-31.51462 & 9.7\\ \hline
27  &   99.64358 &  1.675265  &  -0.08583972& -101.9446& -32.016 &9.2\\ \hline
28  &   99.3585 & 1.651862  &  -0.09177843& -100.8383& -32.52708&8.7\\ \hline
 29 &  99.11297  &  1.628858 &-0.09778215 & -99.77206&-33.04501 &8.2\\ \hline
 30 &  98.90584  & 1.606239  &-0.1038517 &  -98.74333& -33.5708&7.7\\ \hline
 31 &   98.7358 & 1.583989  & -0.1099881& -97.74996& -34.10486 &7.2\\ \hline
32  &  98.6015  & 1.562098  & -0.1161922& -96.78986&-34.64704 & 6.7\\ \hline 
 33 &    98.50187& 1.54056  & -0.122464& -95.86113& -35.19758&6.3\\ \hline
  34&    98.43615 &  1.519339  & -0.1288073&  -94.96198 & -35.75661&5.8\\ \hline
\end{tabular}

\begin{center}
 {Table 1: Upper bounds for expressions in $M_1(k)$.}
\end{center}


\begin{tabular}{c|c|c|c|c|c|c|c}
\hline
k & $A_4(0.001(k+1))$ & $- \ln(z_1(0.001k))$  & $ \ln(z_2(0.001k))$   & $ - \frac{(\xi_1'( 0.001k)-4)^2}{z_1(0.001(k+1))}  $  &  $ - \frac{(\xi_1'(  0.001k)-5)^2}{z_2(0.001k)} $  & $M_3(k)$     \\ \hline

 36 &  98.43379  &  1.477873 & -0.1417047& -93.24588 & -36.90074&4.9\\ \hline
 37 &    98.49569 & 1.457599  & -0.1482617 & -92.42593& -37.48615&4.4\\ \hline 
 38 &  98.58802  & 1.437619  &-0.1548924 & -91.62957& -38.08066&4.0\\ \hline
  39&   98.71033 &  1.417923 & -0.1615978 & -90.85551 &-38.68445 &3.6\\ \hline
 40 &  98.86225  & 1.398503  &-0.1683790 &-90.1025 & -39.29768&3.1\\ \hline
 41 &   99.04347 &  1.379352 & -0.1752370& -89.36971& -39.92054&2.7\\ \hline
  42&   99.25376 & 1.36046  &  -0.1821731& -88.65582& -40.55321&2.3\\ \hline
 43 &    99.49293&  1.341822 & -0.1891883&-87.95995 & -41.1958&1.9\\ \hline
 44 &     99.76085&  1.323429 &  -0.1962838& -87.28121&  -41.84877&1.5\\ \hline
 45 &   100.0576 &  1.305276 & -0.2034610&-86.61873 & -42.51206&1.1\\ \hline
\end{tabular}
\begin{center}
 {Table 2: Upper bounds for expressions in $M_3(k)$.}
\end{center}

\noindent

\vspace{1cm}
\noindent
{\bf \Large Appendix 2: Proof of Lemma~\ref{ph}}\\

By definition, the boundary,  $\partial \Omega$,   of $\Omega$ is
$$\partial \Omega = \{(\chi,\zeta,\xi) \, :\; \xi=0, 2\chi+2\zeta \le 1, 0.454 \le \chi \le 0.45537, \zeta \ge 0\} \cup $$
$$  \{(\chi,\zeta,\xi) \, :\; \zeta=0, 2\chi+\xi \le 1, 0.454 \le \chi \le 0.45537, \xi \ge 0\}\cup$$
$$  \{(\chi,\zeta,\xi) \, :\; \chi=0.454, 2\zeta+\xi \le 0.092,  \zeta \ge 0,\xi \ge 0\} \cup $$
$$   \{(\chi,\zeta,\xi) \, :\; \chi=0.45537, 2\zeta+\xi \le 0.08926,  \zeta \ge 0,\xi \ge 0\}.$$
We also will consider the $2$-dimensional set
 $$\Omega_1= \{(\chi,\zeta) \, :\; 0.454 \le \chi \le 0.45537, 0 \le \zeta \le 0.5-\chi\}.$$
Then the boundary of $\Omega_1$ is
$$\partial \Omega_1 = \{(\chi,\zeta)\, :\; 0.454 \le \chi \le 0.45537, \zeta = 0\} \cup \{(\chi,\zeta)\, :\; 0\le \zeta \le 0.046, \chi = 0.454\} $$
$$\cup \{(\chi,\zeta)\, :\; 0\le \zeta \le 0.04463, \chi = 0.45537\}
 \cup \{(\chi,\zeta)\, :\; 0.454 \le \chi \le 0.45537, \chi+\zeta = \frac{1}{2}\}.$$

By the definition of $h$,
$$\frac{\partial \ln(h(\chi,\zeta,\xi))}{\partial \chi} = 4\ln(2\chi-3\zeta-\frac{1}{2}) + 4\ln(2\chi+\zeta-\frac{1}{2}) - \ln(1-2\chi-2\zeta)$$
$$ - \ln(1-2\chi+4\zeta) + 2\ln(1-2\chi-2\zeta-\xi) + 2\ln(1-2\chi+4\zeta-\xi) - 10\ln(10\chi-5\zeta+\xi-\frac{7}{2}).$$

Similarly to the proof of Lemma~\ref{max}, we present $\frac{\partial \ln(h(\chi,\zeta,\xi))}{\partial \chi}$ 
in the form $\sum_{j=1}^6 B_j,$ where
\begin{equation}\label{B12}
B_1(\chi,\zeta):=4\ln(2\chi-3\zeta-\frac{1}{2}),\quad B_2(\chi,\zeta):=4\ln(2\chi+\zeta-\frac{1}{2}),
\end{equation}
\begin{equation}\label{B34}
B_3(\chi,\zeta,\xi):=2\ln(1-2\chi-2\zeta-\xi)-\ln(1-2\chi-2\zeta), \quad B_4(\chi,\zeta):=\ln(1-2\chi+4\zeta),
\end{equation}
\begin{equation}\label{B56}
B_5(\chi,\zeta,\xi):=2\ln(1-2\chi+4\zeta-\xi), \;\mbox{ and }\; B_6(\chi,\zeta,\xi):=\ln(10\chi-5\zeta+\xi-\frac{7}{2}),
\end{equation}
and then bound each of the terms separately.

\begin{claim}\label{c41} For all $(\chi,\zeta) \in \Omega_1$, $\;B_1(\chi,\zeta) < -3.55.$
\end{claim}

\noindent  {\em Proof.}
For each  $(\chi,\zeta) \in \Omega_1,$ we have $\chi-\frac{3}{2}\zeta-0.25 > 0,$ since $\chi \ge 0.454$ and $\zeta \le 0.046$.
As for each  $(\chi,\zeta) \in \Omega_1,$
$$\frac{\partial B_1(\chi,\zeta)}{\partial \chi} = \frac{4}{\chi - \frac{3}{2}\zeta- 0.25} > 0\qquad
\mbox{\em and }\qquad
\frac{\partial B_1(\chi,\zeta)}{\partial \zeta} = \frac{-6}{\chi - \frac{3}{2}\zeta - 0.25} < 0,$$
the maximum is attained at a corner on the boundary $\partial \Omega_1$. Comparing the values
of $B_1$  at the four corners of $\partial \Omega_1$, we see that the maximum is attained
 at $(\chi,\zeta) = (0.45537,0)$ and $B_1(0.45537,0) < -3.55.$\qed


\begin{claim}\label{c42} For all $(\chi,\zeta) \in \Omega_1$, $\;B_2(\chi,\zeta) < -3.14.$
\end{claim}

\noindent  {\em Proof.}
For each  $(\chi,\zeta) \in \Omega_1,$ we have $2\chi + \zeta - \frac{1}{2} > 0$.
As for each  $(\chi,\zeta) \in \Omega_1,$
$$\frac{\partial B_2(\chi,\zeta)}{\partial \chi} = \frac{8}{2\chi + \zeta - \frac{1}{2}} > 0,
\qquad
\mbox{\em and }\qquad
\frac{\partial B_2(\chi,\zeta)}{\partial \zeta} = \frac{4}{2\chi + \zeta - 0.5} > 0,$$
the maximum is attained at a corner of the boundary $\partial \Omega_1.$
Comparing the values
of $B_2$  at the four corners of $\partial \Omega_1$, we see that 
the maximum is attained at $(\chi,\zeta) = (0.45537,0.04463)$, and $B_2(0.45537,0.04463) < -3.14.$\qed

\begin{claim}\label{c43} For all $(\chi,\zeta,\xi) \in \Omega$, $\;B_3(\chi,\zeta,\xi) < 0.$
\end{claim}

\noindent  {\em Proof.} We can write $B_3(\chi,\zeta,\xi)$ in the form
$$B_3(\chi,\zeta,\xi) = \ln(1-2\chi-2\zeta-\xi)+\ln\left(\frac{1-2\chi-2\zeta-\xi}{1-2\chi-2\zeta}\right),$$
and observe that $\ln(1-2\chi-2\zeta-\xi) < 0$ (since $2\chi + 2\zeta + \xi > 0$)
 and $\ln(\frac{1-2\chi-2\zeta-\xi}{1-2\chi-2\zeta}) \le 0$ (since   $1-2\chi-2\zeta-\xi \le 1-2\chi-2\zeta$)  .\qed

\begin{claim}\label{c44} For all $(\chi,\zeta) \in \Omega_1$, $\;B_4(\chi,\zeta) <-1.28.$
\end{claim}

\noindent  {\em Proof.} For each  $(\chi,\zeta) \in \Omega_1,$
 $-2\chi+4\zeta+1 > 0$. 
As for each  $(\chi,\zeta) \in \Omega_1,$
$$\frac{\partial B_4(\chi,\zeta)}{\partial \chi} = \frac{-2}{-2\chi + 4\zeta + 1} < 0,
\qquad
\mbox{\em and }\qquad
\frac{\partial B_4(\chi,\zeta)}{\partial \zeta} = \frac{4}{-2\chi + 4\zeta + 1} > 0,$$
the maximum of $B_4$ is attained at a corner of the boundary $\partial \Omega_1.$ 
Comparing the values
of $B_4$  at the four corners of $\partial \Omega_1$, we see that 
 the maximum is attained 
at $(\chi,\zeta) = (0.454,0.046)$, and $B_4(0.454,0.046) < -1.28.$\qed

\begin{claim}\label{c45} For all $(\chi,\zeta,\xi) \in \Omega$, $\;B_5(\chi,\zeta,\xi) < -2.57.$
\end{claim}

\noindent  {\em Proof.} 
For each  $(\chi,\zeta,\xi) \in \Omega-\partial \Omega,$  we have $2\chi-4\zeta+\xi-1 < 0$ since $2\chi+2\zeta +\xi < 1 \le 1+ 4\zeta + 2\zeta.$
Since $$\lim\limits_{2\chi+\xi \to 1} B_5(\chi,0,\xi) = - \infty,$$ the maximum of $B_5$ is not  attained at $\zeta = 0,2\chi+\xi = 1.$
As for each  $(\chi,\zeta,\xi) \in \Omega,$
$$\frac{\partial B_5(\chi,\zeta,\xi)}{\partial \chi} = \frac{4}{2\chi-4\zeta+\xi-1} < 0,
\qquad\frac{\partial B_5(\chi,\zeta,\xi)}{\partial \zeta} = \frac{-8}{2\chi-4\zeta+\xi-1} > 0,$$
and $$
\frac{\partial B_5(\chi,\zeta,\xi)}{\partial \xi} = \frac{2}{2\chi-4\zeta+\xi-1} < 0,$$
the maximum of $B_5$ is attained at a corner of the boundary $\partial \Omega.$ 
Comparing the values
of $B_5$  at the  corners of $\partial \Omega$, we see that
 the maximum is attained at $(\chi,\zeta,\xi) = (0.454,0.046,0)$ and $B_5(0.454,0.046,0) < -2.57.$
\qed

\begin{claim}\label{c46} For all $(\chi,\zeta,\xi) \in \Omega$, $\;B_6(\chi,\zeta,\xi) <  0.14.$
\end{claim}

\noindent  {\em Proof.} 
For each  $(\chi,\zeta,\xi) \in \Omega,$  we have $10\chi - 5\zeta + \xi -\frac{7}{2} > 0$ since $10\chi-\frac{7}{2} \ge 1.04$ and $5\zeta \le 0.23.$
As for each  $(\chi,\zeta,\xi) \in \Omega,$
$$\frac{\partial B_6(\chi,\zeta,\xi)}{\partial \chi} = \frac{10}{10\chi - 5\zeta + \xi -\frac{7}{2}} > 0,\qquad
\frac{\partial B_6(\chi,\zeta,\xi)}{\partial \zeta} = \frac{-5}{10\chi - 5\zeta + \xi -\frac{7}{2}} < 0,$$
and
$$\frac{\partial B_6(\chi,\zeta,\xi)}{\partial \xi} = \frac{1}{10\chi - 5\zeta + \xi -\frac{7}{2}} > 0,$$
the maximum of $B_5$ is attained at a corner of the boundary $\partial \Omega.$ 
Comparing the values
of $B_6$  at the  corners of $\partial \Omega$, we see that
 the maximum is attained at $(\chi,\zeta,\xi) = (0.45537,0,0.08926)$ and $B_6(0.454,0,0.08926) < 0.14.$
\qed

By Claims~\ref{c41}--\ref{c46},  for each  $(\chi,\zeta,\xi) \in \Omega,$
$$\frac{\partial \ln(h(\chi,\zeta,\xi))}{\partial \chi} = B_1(\chi,\zeta) + B_2(\chi,\zeta) + B_3(\chi,\zeta,\xi) + B_4(\chi,\zeta) + B_5(\chi,\zeta,\xi) + B_6(\chi,\zeta,\xi)  $$
$$<-3.55-3.14+0-1.28-2.57+0.14  
 < 0.\qed $$




\end{document}